\documentclass[review]{elsarticle}
\usepackage{lineno,hyperref}
\modulolinenumbers[5]

\usepackage{amsmath}
\usepackage{graphicx}
\usepackage{epstopdf}
\usepackage{amsfonts}
\usepackage{verbatim}
\usepackage{latexsym}

\makeatletter
\@addtoreset{equation}{section}

\begin{document}

\newcommand{\E}{\mathbb{E}}
\newcommand{\PP}{\mathbb{P}}
\newcommand{\RR}{\mathbb{R}}
\newcommand{\SM}{\mathbb{S}}

\newtheorem{theorem}{Theorem}[section]
\newtheorem{lemma}[theorem]{Lemma}
\newtheorem{coro}[theorem]{Corollary}
\newtheorem{defn}[theorem]{Definition}
\newtheorem{assp}[theorem]{Assumption}
\newtheorem{cond}[theorem]{Condition}
\newtheorem{expl}[theorem]{Example}
\newtheorem{prop}[theorem]{Proposition}
\newtheorem{rmk}[theorem]{Remark}
\newtheorem{conj}[theorem]{Conjecture}

\newcommand\tq{{\scriptstyle{3\over 4 }\scriptstyle}}
\newcommand\qua{{\scriptstyle{1\over 4 }\scriptstyle}}
\newcommand\hf{{\textstyle{1\over 2 }\displaystyle}}
\newcommand\hhf{{\scriptstyle{1\over 2 }\scriptstyle}}
\newcommand\hei{\tfrac{1}{8}}

\newcommand{\eproof}{\indent\vrule height6pt width4pt depth1pt\hfil\par\medbreak}

\def\a{\alpha} \def\g{\gamma}
\def\e{\varepsilon} \def\z{\zeta} \def\y{\eta} \def\o{\theta}
\def\vo{\vartheta} \def\k{\kappa} \def\l{\lambda} \def\m{\mu} \def\n{\nu}
\def\x{\xi}  \def\r{\rho} \def\s{\sigma}
\def\p{\phi} \def\f{\varphi}   \def\w{\omega}
\def\q{\surd} \def\i{\bot} \def\h{\forall} \def\j{\emptyset}

\def\be{\beta} \def\de{\delta} \def\up{\upsilon} \def\eq{\equiv}
\def\ve{\vee} \def\we{\wedge}

\def\F{{\cal F}}
\def\T{\tau} \def\G{\Gamma}  \def\D{\Delta} \def\O{\Theta} \def\L{\Lambda}
\def\X{\Xi} \def\Si{\Sigma} \def\W{\Omega}
\def\M{\partial} \def\N{\nabla} \def\Ex{\exists} \def\K{\times}
\def\V{\bigvee} \def\U{\bigwedge}

\def\1{\oslash} \def\2{\oplus} \def\3{\otimes} \def\4{\ominus}
\def\5{\circ} \def\6{\odot} \def\7{\backslash} \def\8{\infty}
\def\9{\bigcap} \def\0{\bigcup} \def\+{\pm} \def\-{\mp}
\def\la{\langle} \def\ra{\rangle}

\def\proof{\noindent{\it Proof. }}
\def\tl{\tilde}
\def\trace{\hbox{\rm trace}}
\def\diag{\hbox{\rm diag}}
\def\for{\quad\hbox{for }}
\def\refer{\hangindent=0.3in\hangafter=1}

\newcommand\wD{\widehat{\D}}
\newcommand{\ka}{\kappa_{10}}

\begin{frontmatter}
\title{Mean Square Polynomial Stability of Numerical Solutions to a Class of Stochastic Differential Equations}

\author[WLMF]{Wei Liu\corref{cor}}
\ead{w.liu@lboro.ac.uk;lwbvb@hotmail.com}

\author[WLMF]{Mohammud Foondun}

\author[XM]{Xuerong Mao}

\cortext[cor]{Corresponding author}

\address[WLMF]{Loughborough University, Department of Mathematical Sciences, Loughborough, Leicestershire, LE11 3TU, UK}

\address[XM]{University of Strathclyde, Department of Mathematics and Statistics, Glasgow, G1 1XH, UK}

\begin{abstract}
The exponential stability of numerical methods to stochastic differential equations (SDEs) has been widely studied. In contrast, there are relatively few works on polynomial stability of numerical methods. In this letter, we address the question of reproducing the polynomial decay of a class of SDEs using the Euler--Maruyama method and the backward Euler--Maruyama method. The key technical contribution is based on various estimates involving the gamma function.
\end{abstract}
\end{frontmatter}

\linenumbers

\section{Introduction} \label{intro}
The stability of stochastic differential equations (SDEs) has been widely studied by many authors, see for instance \cite{Pro1990a,Liu2006a,M2008a,Kha2012} and references therein. In particular, different decay rates have been widely investigated, for example, exponential stability \cite{LMS2006a,MYY2007a}, polynomial stability \cite{Mao1992poly,AB2003a} and general rate \cite{LC2001a,CGR2003a}.
\par
It is natural to ask whether numerical solutions of SDEs preserve stability properties of the original SDEs and in recent years, this question has received quite a lot of attention.  See for instance \cite{Sai96,Sch1997a,Hig2000a} for some of the original ideas in this area. An important feature of these works is that they focus on the exponential stability of the numerical solutions \cite{HMY2007a,Pan08,WG2011a,WMK2013a}. 
\par
As far as we know, there are few papers devoted to the polynomial stability of the numerical solutions. In this letter, we try to fill the gap by studying the reproduction of the polynomial stability of a class of SDEs by using the Euler-Maruyama (EM) method and the backward Euler-Maruyama (BEM) method. We mention some related works on difference equations \cite{AGR2007a,AMR2008a}. Our approach is different from theirs.
\par
Since we are interested in reproducing polynomial decay rates, we need to use different techniques to those used to handle the exponential decay rates. Our method hinges on various properties of the gamma function and ratios of gamma functions. This is significantly different from say \cite{HMY2007a}, where exponential stability is considered. 

\par
The structure of the letter is as follows. In Section \ref{mainres}, we give some background information which include the polynomial stability of the class of SDEs.  We then state and prove two lemmas about the gamma function, both of which will be crucial in the proofs of our main results.  In section \ref{themain} we first study the polynomial stability of the EM method followed by an counterexample which shows the failure of this method to preserve polynomial stability when an superlinear term appears on the drift coefficient. We then show that by using a semi-implicit method, namely the BEM, one can indeed preserve polynomial stability under less stringent conditions.
We illustrate our result with some simulations in  Section \ref{numexp} and conclude the letter with a discussion of future research in Section \ref{conclusion}. 

\section{Preliminary}\label{mainres}
Throughout this letter, we let ($\Omega,\F, \{\F_t\}_{t \geq 0}, \PP$) be a complete probability space with a filtration $\{\F_t\}_{t \geq 0}$ that is increasing and right continuous, with $\F_0$ containing all $\PP$-null sets. Let $|\cdot|$ denote the Euclidean norm in $\RR^n$. The inner product of $x$ and $y$ in $\RR^n$ is denoted by $\la x, y \ra$. For any $a \in \RR$, $[a]$ denotes the integer part of a. For the sake of simplicity, we only consider the case of scalar Brownian motion denoted by $B(t)$ and defined on the probability space.  But it is not hard to see that our results can be extended to the case of multi-dimensional Brownian motion.
\par
In this letter, we study the SDEs of It\^o type,
\begin{equation}
\label{sde}
dx(t) = f(x(t),t) dt + g(x(t),t)dB(t) ~~~ \text{on} ~~~ t \geq 0,
\end{equation}
with initial value $x(0) \in  \RR^n$, and $f,g$ : $\RR^n \rightarrow \RR^n$.  We will need following conditions for later.

\begin{cond}
For every integer $r \geq 1$ and any $t \geq 0$, there exists a positive constant $\bar{K}_{r,t}$ such that, for $\forall x, y \in \RR^n$ with $\max(|x|,|y|) \leq r$,
\begin{equation}
\label{loclip}
\max (|f(x,t) - f(y,t)|^2,|g(x,t) - g(y,t)|^2) \leq \bar{K}_{r,t} |x - y|^2.
\end{equation}
\end{cond}
This condition is needed for the existence and uniqueness of the solution, so we do not need any more information about how the constant depends on $t$. 
\begin{cond}\label{onesidedlipc}
We assume that for $\forall x, y \in \RR^n$ and $\forall t > 0$, there exists a constant $\bar{K}$ such that
\begin{equation}
\label{onesidedlip}
\la x-y, f(x,t) -f(y,t) \ra \leq \bar{K} (1 + t)^{-1} |x - y|^2.
\end{equation}
\end{cond}

\begin{cond}\label{growth}
Suppose that $\forall t>0$ and $x\in \RR^d$, there exist positive constants $K_1$ and $C$ such that
\begin{equation}
\label{lingrof}
|f(x,t)| \leq K_1 (1+t)^{-1} |x|,
\end{equation}

\begin{equation}
\label{onesidf}
\la x, f(x,t) \ra \leq -K_1 (1+t)^{-1} |x|^2,
\end{equation}

\begin{equation}
\label{lingrog}
|g(x,t)| \leq C (1+t)^{-K_1}.
\end{equation}
\end{cond}

The following theorem states that under some condition, the solution to \eqref{sde} is mean-square polynomially stable.  Our subsequent numerical approximations will seek to preserve the upper bound on the rate stated in this theorem. 
  
\begin{theorem}
\label{themsde}
Suppose that \eqref{onesidf}and \eqref{lingrog} of Condition \ref{growth} hold, if $K_1 > 0.5$ then the solution to (\ref{sde}) is mean square polynomially stable. In other words, for any initial value $x(0) \in \RR^n$

\begin{equation*}
\limsup_{t \rightarrow \infty} \frac{\log \E |x(t)|^2 }{\log t} \leq - (2 K_1 - 1 ).
\end{equation*}

\end{theorem}
\par
\noindent
This is well known so we refer the readers to  \cite{LC2001a} for the proof.
\par
\noindent
\medskip
The next lemma benefits from the following infinite product definition for the gamma function. Suppose that $x$ is any real number, except non-positive integers, then
\begin{equation}
\label{gammainfprod}
\Gamma(x) =  \frac{1}{x} \prod_{j=1}^{\infty}\frac{(1+\frac{1}{j})^x}{1+\frac{x}{j}}.
\end{equation}
See, for example, Page 894 of \cite{JZ2007a} for more details. 

\begin{lemma}
\label{prodtogamma}
Given $\alpha > 0$ and $\beta \geq 0$, if there exists a $\delta$ such that $0 < \delta < \alpha^{-1}$ then
\begin{equation*}
\prod_{i = a}^{b} \left( 1 - \frac{\alpha \delta}{1 + (i + \beta) \delta} \right) = \frac{\Gamma (b + 1 + \delta^{-1} + \beta - \alpha)}{\Gamma(b + 1 +  \delta^{-1} + \beta) } \times \frac{\Gamma(a +  \delta^{-1} + \beta)}{\Gamma( a +  \delta^{-1} + \beta - \alpha)},
\end{equation*}
where $0 \leq a \leq b$.
\end{lemma}

\begin{proof}
One can rewrite the finite product into the ratio of two infinite products,
\begin{equation}
\label{ratioexpress}
\prod_{i = a}^{b} \left( 1 - \frac{\alpha \delta}{1 + (i + \beta) \delta} \right) = \frac{\prod_{i = a}^{\infty} \left( 1 - \frac{\alpha \delta}{1 + (i + \beta) \delta} \right)}{\prod_{i = b+1}^{\infty} \left( 1 - \frac{\alpha \delta}{1 + (i + \beta) \delta} \right)}.
\end{equation} 
For the numerator of (\ref{ratioexpress}), one can further rewrite it into that
\begin{eqnarray}
\label{numexpress}
&&\prod_{i = a}^{\infty} \left( 1 - \frac{\alpha \delta}{1 + (i + \beta) \delta} \right) 
\nonumber \\
&=& \frac{1 + (a + \beta - \alpha)\delta}{ 1 + (a + \beta)\delta} \prod_{i = 1}^{\infty} \left( 1 - \frac{\alpha \delta}{1 + (i + a + \beta) \delta} \right)
\nonumber \\
&=& \frac{1 + (a + \beta - \alpha)\delta}{ 1 + (a + \beta)\delta}  \prod_{i = 1}^{\infty} \left( \frac{i \delta}{1 + (i + a + \beta) \delta} \times \frac{1 + (i + a + \beta - \alpha) \delta}{i \delta} \right)
\nonumber \\
&=& \frac{a + \delta^{-1} + \beta - \alpha}{a+ \delta^{-1} + \beta} \prod_{i = 1}^{\infty} \left( \frac{1}{ 1 + \frac{a+ \delta^{-1} + \beta}{i}} \times \left( 1 + \frac{a + \delta^{-1} + \beta - \alpha}{i} \right) \right).
\end{eqnarray}
Similarly, one can rewrite the denominator of (\ref{ratioexpress}) into that
\begin{eqnarray}
\label{denexpress}
&&\prod_{i = b+1}^{\infty} \left( 1 - \frac{\alpha \delta}{1 + (i + \beta) \delta} \right)
\nonumber \\
&=& \left ( \frac{b + 1 + \delta^{-1} + \beta - \alpha}{ b + 1 + \delta^{-1} + \beta} \right )
\nonumber \\
&&\times \prod_{i = 1}^{\infty} \left( \frac{1}{ 1 + \frac{b + 1 + \delta^{-1} + \beta}{i}} \times \left( 1 + \frac{b + 1 + \delta^{-1} + \beta - \alpha}{i} \right) \right).
\end{eqnarray}
Substituting (\ref{numexpress}) and (\ref{denexpress}) into (\ref{ratioexpress}), and using the identity that
\begin{equation*}
1 = \prod_{i = 1}^{\infty} \left(1 + \frac{1}{i} \right)^{(a + \delta^{-1} + \beta) - (a + \delta^{-1} + \beta - \alpha) + (b + 1 + \delta^{-1} + \beta - \alpha) - (b + 1 + \delta^{-1} + \beta)},
\end{equation*}
one can see that
\begin{eqnarray*}
&&\prod_{i = a}^{b} \left( 1 - \frac{\alpha \delta}{1 + (i + \beta) \delta} \right) 
\nonumber \\
&=& \left( \frac{1}{b + 1 + \delta^{-1} + \beta - \alpha} \prod_{i = 1}^{\infty} \frac{(1 + \frac{1}{i})^{b + 1 + \delta^{-1} + \beta - \alpha}}{1 + \frac{b + 1 + \delta^{-1} + \beta - \alpha}{i}} \right) \times  \left( \frac{1}{b + 1 + \delta^{-1} + \beta} \prod_{i = 1}^{\infty} \frac{(1 + \frac{1}{i})^{b + 1 + \delta^{-1} + \beta}}{1 + \frac{b + 1 + \delta^{-1} + \beta}{i}} \right)^{-1}
\nonumber \\
&&\times \left( \frac{1}{a + \delta^{-1} + \beta} \prod_{i = 1}^{\infty} \frac{(1 + \frac{1}{i})^{a + \delta^{-1} + \beta}}{1 + \frac{a + \delta^{-1} + \beta}{i}} \right) \times \left( \frac{1}{a + \delta^{-1} + \beta - \alpha} \prod_{i = 1}^{\infty} \frac{(1 + \frac{1}{i})^{a + \delta^{-1} + \beta - \alpha}}{1 + \frac{a + \delta^{-1} + \beta - \alpha}{i}} \right)^{-1}.
\end{eqnarray*}
Applying (\ref{gammainfprod}), the proof is complete. \eproof 
\end{proof}
\medskip
\par
\noindent
Next, we present two estimates about the ratio of the gamma functions in the following lemma, and refer the readers for more results about the ratio of the gamma functions to \cite{KER1983}. The technique used in the proof of the following lemma is similar to that in \cite{KER1983}, so we just brief the proof. 
\begin{lemma}
\label{lemgamma}
For any $x>0$, if $0 < \eta < 1$
\begin{equation}
\label{tlessone}
\frac{\Gamma (x+\eta)}{\Gamma(x)} < x^\eta,
\end{equation}
and if $\eta >1$
\begin{equation}
\label{tlargerone}
\frac{\Gamma (x+\eta)}{\Gamma(x)} > x^\eta.
\end{equation}
\end{lemma}

\begin{proof}
Define $h(x)$ by
\begin{equation*}
h(x) = \frac{\Gamma (x+\eta)}{\Gamma(x)} x^{-\eta}, 
\end{equation*}
from page 12 of \cite{MOS1966} one has that for any $\eta > 0$
\begin{equation*}
\lim_{x \rightarrow \infty} h(x) = 1.
\end{equation*}
Now let
\begin{equation*}
H(x) = \frac{h(x)}{h(x+1)} = \frac{x}{x+\eta}\left( \frac{x+1}{x}\right)^\eta.
\end{equation*}
Then one sees that
\begin{equation*}
\frac{H'(x)}{H(x)} = \frac{\eta(1-\eta)}{x(x+\eta)(x+1)}.
\end{equation*}
If $0 < \eta <1$ one has $H'(x) > 0$, then $H(x)$ is a strict increasing function. Since $\lim_{x \rightarrow \infty} H(x) = 1$, one has that $H(x) < 1$ for all $x > 0$, which indicates $h(x) < h(x + 1) < h(x + n)$. Letting $n \rightarrow \infty$ yields $h(x) < 1$, that is (\ref{tlessone}) holds.
\par
If $\eta > 1$, one has $H'(x) < 0$ and $H(x)$ is a strict decreasing function. Using the similar argument as above, one sees that $H(x) > 1 $ and $h(x) > 1$. Thus (\ref{tlargerone}) holds. \eproof
\end{proof}

\section{Main Results}\label{themain}
We are now ready to state and prove the main results of the letter. We begin with the EM method.

\subsection{The Euler--Maruyama Method}
The EM method for (\ref{sde}) is given by
\begin{equation}
\label{EMdis}
Y_{k+1} = Y_k + f(Y_k, k \D t ) \D t + g(Y_k, k \D t) \Delta B_k,~~~Y_0 = x(0),
\end{equation}
where $\D t > 0$ is the time step  and $\D B_k = B((k+1)\D t) - B(k \D t)$ is the Brownian motion increment. Note that for the result below to remain valid we require \eqref{lingrof} of Condition \ref{growth} to hold.
\par

\begin{theorem}
\label{themEM}
Assume \eqref{lingrof}, \eqref{onesidf} and \eqref{lingrog} of Condition \ref{growth}.  If $K_1 \geq 1$ and $\D t < (2+K_1)^{-1}$ then the EM solution (\ref{EMdis}) satisfies

\begin{equation}
\label{EMmeansquare}
\limsup_{k \rightarrow \infty} \frac{\log \E |Y_k|^2}{\log k \D t} \leq - (2 K_1 - 1),
\end{equation}
for any initial value $Y_0 \in \RR^n$.
\end{theorem}

\begin{proof}
Taking square on both sides of (\ref{EMdis}) yields
\begin{equation*}
|Y_{k+1}|^2 = | Y_k + f(Y_k, k \D t ) \D t |^2 + | g(Y_k, k \D t) \Delta B_k |^2 + 2 \la Y_k + f(Y_k, k \D t ) \D t, g(Y_k, k \D t) \Delta B_k \ra.
\end{equation*}
Then taking expectation on both sides and using (\ref{lingrof}), (\ref{onesidf}) and (\ref{lingrog}), one can see that
\begin{eqnarray*}
\E |Y_{k+1}|^2 &\leq& \E |Y_k|^2 - \frac{2 K_1 \D t \E |Y_k|^2}{ 1 + k \D t} + \frac{K_1^2 \D t^2  \E |Y_k|^2 }{(1 + k \D t)^2} +\frac{C^2 \D t}{(1 + k \D t)^{2 K_2}}
\nonumber \\
&\leq& \left( 1 - \frac{K_1\Delta t}{1+k\Delta t} \right)^2 \E|Y_k|^2 + C^2(1 + k\Delta t)^{-2K_1} \Delta t.
\end{eqnarray*}
From the iteration above, one can derive the relation between $\E|Y_k|^2$ and $\E|Y_0|^2$ that
\begin{eqnarray}
\label{EMitori}
\E|Y_k|^2 &\leq& \left( \prod_{i=0}^{k-1}\left( 1-\frac{K_1\Delta t}{1+ i \Delta t}\right)^2 \right) \E|Y_0|^2 
\nonumber \\
&&+ \sum_{r=0}^{k-1} \left( \prod_{i=r+1}^{k-1}\left( 1-\frac{K_1 \Delta t}{1+i \Delta t}\right)^2 \right) (1 +r\Delta t)^{-2K_1} C^2 \Delta t,
\end{eqnarray}
where one assumes $\prod_{i=k}^{k-1} (1 - K_1 \D t (1 + i \D t )^{-1})^2 = 1$ for the convenience of the notation.
\par
Applying Lemma \ref{prodtogamma} to (\ref{EMitori}), one can simplify (\ref{EMitori}) to get that
\begin{eqnarray}
\label{EMitgamma}
\E|Y_k|^2 &\leq& \frac{\Gamma (k + \frac{1}{\Delta t} - K_1)^2 \Gamma (\frac{1}{\Delta t})^2}{\Gamma (k+\frac{1}{\Delta t})^2 \Gamma (\frac{1}{\Delta t}-K_1)^2} \E|Y_0|^2
\nonumber \\
&&+ \sum_{r=0}^{k-1} \frac{\Gamma (k + \frac{1}{\Delta t} - K_1)^2 \Gamma (r + 1+ \frac{1}{\Delta t})^2}{\Gamma (k+\frac{1}{\Delta t})^2 \Gamma ( r+ 1+\frac{1}{\Delta t} - K_1)^2} (1+r \Delta t)^{-2K_1} C^2 \Delta t.
\end{eqnarray}

One considers the first term on the right hand side of the equality in (\ref{EMitgamma}). Since $K_1 \geq 1$, by (\ref{tlargerone})  one can see that 
\begin{equation*}
\frac{\Gamma (k+\frac{1}{\Delta t})}{\Gamma (k + \frac{1}{\Delta t} - K_1)} = \frac{\Gamma ((k + \frac{1}{\Delta t} - K_1) + K_1)}{\Gamma (k + \frac{1}{\Delta t} - K_1)} \geq \left(k + \frac{1}{\Delta t} - K_1\right)^{K_1}.
\end{equation*}
Due to the fact that 
\begin{equation*}
\Gamma \left(\frac{1}{\Delta t}\right) = \prod_{i=1}^{[K_1]} \left(\frac{1}{\Delta t} - i \right) \Gamma \left(\frac{1}{\Delta t} - [K_1]\right),
\end{equation*}
where the iteration $\Gamma(x+1) = x\Gamma(x)$ is applied repeatedly, and (\ref{tlessone}), one can see 
\begin{equation*}
\frac{\Gamma (\frac{1}{\Delta t})}{\Gamma (\frac{1}{\Delta t} - K_1)} = \frac{\Gamma (\frac{1}{\Delta t} - [K_1])}{\Gamma (\frac{1}{\Delta t} - K_1)}\prod_{i=1}^{[K_1]} \left(\frac{1}{\Delta t} - i \right) \leq \left( \frac{1}{\D t} - K_1 \right)^{K_1 - [K_1]}  \left(\frac{1}{\Delta t} \right)^{[K_1]} \leq  \left(\frac{1}{\Delta t} \right)^{K_1}.
\end{equation*}
So one can obtain that 
\begin{equation}
\label{firupp}
\frac{\Gamma (k + \frac{1}{\Delta t} - K_1)^2 \Gamma (\frac{1}{\Delta t})^2}{\Gamma (k+\frac{1}{\Delta t})^2 \Gamma (\frac{1}{\Delta t} - K_1)^2} \leq (k + \frac{1}{\Delta t} - K_1)^{-2 K_1} \left(\frac{1}{\Delta t} \right)^{2K_1} \leq ((k -K_1)\D t + 1)^{-2K_1}.
\end{equation}
Next, one considers the second term on the right hand side of the equality in (\ref{EMitgamma}). Due to the fact that
\begin{equation*}
\Gamma \left(r + 1 + \frac{1}{\D t} \right) = \prod_{i=1}^{[K_1]} \left(r + 1 + \frac{1}{\D t} - i \right) \Gamma \left(r + 1 + \frac{1}{\D t} - [K_1] \right),
\end{equation*}
and (\ref{tlessone}), one can get that
\begin{eqnarray*}
\frac{\Gamma (r + 1 + \frac{1}{\D t} ) }{\Gamma (r + 1 + \frac{1}{\D t} - K_1)} &\leq& \left(r + 1 + \frac{1}{\D t} - K_1\right)^{K_1 - [K_1]} \prod_{i=1}^{[K_1]} \left(r + 1 + \frac{1}{\D t} - i \right) 
\nonumber \\
&\leq&  \left(r + 1 + \frac{1}{\D t} \right)^{K_1}
\nonumber \\
&\leq& ((r+1)\D t + 1)^{K_1} \D t^{- K_1}.
\end{eqnarray*}
So one can see that
\begin{eqnarray}
\label{secupp}
\frac{\Gamma (k + \frac{1}{\Delta t} - K_1)^2 \Gamma (r + 1+ \frac{1}{\Delta t})^2}{\Gamma (k+\frac{1}{\Delta t})^2 \Gamma ( r+ 1+\frac{1}{\Delta t} - K_1)^2} &\leq& \left(k + \frac{1}{\Delta t} - K_1\right)^{-2 K_1}((r+1)\D t + 1)^{2K_1} \D t^{- 2K_1}
\nonumber \\
&\leq& ((k -K_1)\D t + 1)^{-2K_1} ((r+1)\D t + 1)^{2K_1}.
\end{eqnarray}
Substituting (\ref{firupp}) and (\ref{secupp}) into (\ref{EMitgamma}) yields
\begin{eqnarray*}
\E|Y_k|^2 &\leq&  ((k -K_1)\D t + 1)^{-2K_1} \left( \E|Y_0|^2 + C^2 \Delta t \sum_{r=0}^{k-1} ((r+1)\D t + 1)^{2K_1}(1+r \Delta t)^{-2K_1}  \right)
\nonumber \\
&\leq& ((k -K_1)\D t + 1)^{-2K_1} \left( \E|Y_0|^2 + C^2 C_1^{2K_1} k \D t \right)
\nonumber \\
&\leq& (k \D t + 1)^{- 2 K_1 + 1} (\E|Y_0|^2  + C^2 C_1^{2K_1}),
\end{eqnarray*}
where the fact that $(1+(r+1) \Delta t)/(1+ r \Delta t)$ is bounded by some positive constant $C_1$ is used. Therefore, the assertion (\ref{EMmeansquare}) holds. \eproof
\end{proof}

\begin{rmk}
It can be noticed that for the EM method in the theorem above we need $K_1 \geq 1$ and this is due to the application of the estimate (\ref{tlargerone}), but for the SDE in Theorem \ref{themsde} we only need $K_1 > 0.5$. That is to say that although the EM method can reproduce the polynomial stability of the SDE, the result is not sharp.   
\end{rmk}
\subsection{A Counterexample }
Let us consider the following scalar SDE,
\begin{equation}
\label{conexp}
dx(t) = \frac{- 3 x(t) - x^3(t)}{1+ t} dt + \frac{1}{(1+t)^3} dB(t).
\end{equation}
It is not difficult to check that \eqref{onesidf} and \eqref{lingrog} hold but not \eqref{lingrof}. By Theorem \ref{themsde}, we know that the underlying solution to (\ref{conexp}) is mean square polynomial stable. But the following lemma shows that for any given initial value, the EM solution will blow up as time advances, which contracts to the initial-value independent
stability of the underlying SDE. 

\begin{lemma}
Suppose $\D t \in (0,0.5)$, the for any $Y_0 \in \RR$,
\begin{equation*}
\lim_{k \rightarrow \infty} \E |Y_k| = \infty.
\end{equation*}
\end{lemma}

\begin{proof}
Using the property of conditional expectations, one has that
\begin{equation}
\label{mainexparg}
\E |Y_{k+1}| = \E (\E(|Y_{k+1}| \big\vert Y_1)) \geq \E (\mathbf{1}_{\{|Y_1| \geq 3(\D t/(1+ \D t))^{-0.5}\}} \E (|Y_{k+1}| \big\vert Y_1)).
\end{equation}
As there exists a nonzero probability that the first Brownian motion increment will yield $|Y_1| \geq 3(\D t/(1+ \D t))^{-0.5}$, one only needs to prove that $\E (|Y_{k+1}| \big\vert Y_1) \geq (\D t/(1+ (k+1)\D t))^{-0.5} (k+3)$ for all $k \geq 0$ given $|Y_1| \geq (\D t/(1+ \D t))^{-0.5}$. One shows this by induction. Obviously, $\E (|Y_1| \big\vert Y_1) = |Y_1| \geq 3(\D t/(1+ \D t))^{-0.5}$. Suppose $E (|Y_{k}| \big\vert Y_1) \geq (\D t/(1+ k\D t))^{-0.5} (k+2)$ for some $k \geq 1$, one will show that for any $\D t \in (0,0.5)$, $\E (|Y_{k+1}| \big\vert Y_1) \geq (\D t/(1+ (k+1)\D t))^{-0.5} (k+3)$. Applying the EM method to the SDE (\ref{conexp}), one has that
\begin{equation*}
|Y_{k+1}| = \left |Y_k + \frac{-3 Y_k - Y_k^3}{1 + k \D t} \D t + \frac{1}{(1+k \D t)^3} \D B_k \right |.
\end{equation*} 
By the elementary inequality, one can see that
\begin{eqnarray}
\label{interineone}
|Y_{k+1}|  &\geq& \left | \frac{\D t}{1+k \D t} Y_k^3  \right | - \left (1 - \frac{3 \D t}{1+k \D t} \right) |Y_k| - \frac{1}{(1+k \D t)^3} |\D B_k|
\nonumber \\
&\geq& \left | \frac{\D t}{1+k \D t} Y_k^3  \right | - |Y_k| - |\D B_k|.
\end{eqnarray}
Thanks to H\"older's inequality, one has $\E (|Y_k|^3 \big \vert Y_1) \geq (\E (|Y_k| \big \vert Y_1))^3$. Since $\D B_k$ is independent of $Y_1$ for all $k > 0$, one has $\E (|\D B_k| \big \vert Y_1) = \E (|\D B_k|) < 1$. Taking conditional expectation on both sides of (\ref{interineone}) one has that
\begin{eqnarray*}
\E (|Y_{k+1}| \big \vert Y_1) &\geq& \frac{\D t}{1+k \D t} (\E (|Y_k| \big \vert Y_1))^3 - \E (|Y_k|\big \vert Y_1) - 1
\nonumber \\
&\geq& \E |Y_k| \left ( \frac{\D t}{1+k \D t} (\E ( |Y_k| \big \vert Y_1))^2 - 1\right ) - 1
\nonumber \\
&\geq& \left ( \frac{\D t}{1 + k \D t} \right )^{-0.5} (k+2) \left ( (k+2)^2 - 1 \right ) -1
\nonumber \\
&\geq& \left ( \frac{\D t}{1 + (k+1) \D t} \right )^{-0.5} \left ( \frac{1 + k\D t}{1 + (k+1) \D t} \right )^{0.5} ((k+2)^3 - (k+2)) - 1
\nonumber \\
&\geq& \left ( \frac{\D t}{1 + (k+1) \D t} \right )^{-0.5} (k+3).
\end{eqnarray*}
Then substituting this back to (\ref{mainexparg}), one obtains that 
\begin{equation*}
\E |Y_{k+1}|  \geq \left ( \frac{\D t}{1 + (k+1) \D t} \right )^{-0.5} (k+3) \PP \left ( |Y_1| \geq 3 \left (\frac{\D t}{1+ \D t} \right )^{-0.5} \right ).
\end{equation*}
Therefore the assertion holds. \eproof
\end{proof}
\noindent
\par

For more information about examples of these type, we refer the readers to \cite{HMY2007a,HJK2011a}.  The backward Euler-Maruyama method which is a semi-implicit method is a good replacement of the EM method in this situation. 

\subsection{The Backward Euler--Maruyama Method}
The BEM method for the SDE (\ref{sde}) is defined by
\begin{equation}
\label{BEM}
Z_{k+1} = Z_k + f(Z_{k+1}, (k+1) \D t) \D t + g(Z_k,k\D t) \D B_k, ~~~Z_0 = x(0).
\end{equation}

Since BEM is a semi-implicit method, we will need the following result for the method to be well defined. This result basically established the existence of a solution to a certain equation which is needed for the BEM to makes sense.

\begin{lemma}
Suppose that Condition \ref{onesidedlipc} and (\ref{onesidf}) of Conditon \ref{growth} hold and that $\D t < |\bar{K}|^{-1}$. Then for any $t > 0$ and $b \in \RR^n$, there is a unique root $x \in \RR^n$ of the equation
\begin{equation*}
x = f(x,t) \D t + b.
\end{equation*}
\end{lemma}
\begin{proof}
Define $F(x,t) = x - f(x,t) \D t$. By (\ref{onesidedlip}) one can see that
\begin{equation*}
\la x - y, F(x,t) - F(y,t) \ra \geq | x - y|^2 -\bar{K} (1+ t)^{-1} \D t |x - y|^2 \geq (1 - |\bar{K}| \D t) |x-y|^2 > 0,
\end{equation*}
for $\D t < |\bar{K}|^{-1}$, i.e. $F(x)$ is monotone. In addition, one can derive from (\ref{onesidf}) that 
\begin{equation*}
\la x , F(x,t) \ra \geq (1 + K_1 \D t (1+ t)^{-1}) |x|^2,
\end{equation*}
which indicates 
\begin{equation*}
\lim_{|x| \rightarrow \infty} \frac{\la x , F(x,t) \ra}{|x|} = \infty.
\end{equation*}
We can conclude that here is a unique root $x \in \RR^n$ of the equation. See for instance Lemma 3.1 in \cite{MS2013a}.
\end{proof}

\begin{theorem}
\label{themBEM}
Suppose that  (\ref{onesidf}) and (\ref{lingrog}) of Condition \ref{growth} hold.  If $K_1 > 0.5$ and  $\D t <  \min(|\bar{K}|^{-1}, K_1^{-1})$ then the BEM solution (\ref{BEM}) satisfies
\begin{equation}
\label{BEMmeansquare}
\limsup_{k \rightarrow \infty} \frac{\log \E |Z_k|^2}{\log k \D t} \leq - (2 K_1 - 1),
\end{equation}
for any intial value $Z_0 \in \RR^n$.
\end{theorem}

\begin{proof}
From (\ref{BEM}), one has 
\begin{equation*}
|Z_{k+1}|^2 = \la Z_{k+1}, Z_k + g(Z_k,k\D t) \D B_k \ra + \la Z_{k+1} , f(Z_{k+1}, (k+1) \D t) \D t \ra.
\end{equation*}
By the elementary inequality and (\ref{onesidf}) , one can see that
\begin{equation*}
|Z_{k+1}|^2 \leq 0.5 |Z_{k+1}|^2 + 0.5 | Z_k + g(Z_k,k\D t) \D B_k |^2 - K_1 \D t (1+(k+1)\D t)^{-1} |Z_{k+1}|^2.
\end{equation*}
Simplifying the inequality above by putting all terms of $Z_{k+1}$ on left hand side of the inequality, one has
\begin{equation*}
|Z_{k+1}|^2 \leq \left( 1 - \frac{2K_1\D t}{1 + (k+1) \D t + 2 K_1 \D t} \right) | Z_k + g(Z_k,k\D t) \D B_k |^2.
\end{equation*} 

Taking expectations on both sides and using (\ref{lingrog}) yields
\begin{equation*}
\E |Z_{k+1}|^2 \leq \left( 1 - \frac{2K_1\D t}{1 + (k+1) \D t + 2 K_1 \D t} \right) \left( \E |Z_k|^2 + C^2 \D t (1+k \D t)^{- 2 K_1} \right).
\end{equation*}
Now by the iteration, one has that
\begin{eqnarray*}
\E |Z_k|^2 &\leq& \prod_{i=1}^{k} \left( 1 - \frac{2K_1\D t}{1 + i \D t + 2 K_1 \D t} \right) \E |Z_0|^2
\nonumber \\
&&+ C^2 \D t \sum_{r=0}^{k-1} \prod_{i=r+1}^{k} \left( 1 - \frac{2K_1\D t}{1 + i \D t + 2 K_1 \D t} \right) (1 + r \D t)^{- 2 K_1}.
\end{eqnarray*}
Applying Lemma \ref{prodtogamma}, one can rewrite the inequality above into 
\begin{eqnarray}
\label{BEMmain}
\E |Z_k|^2 &\leq& \frac{\Gamma (k + 1 + \frac{1}{\D t})  \Gamma (1 + 2 K_1 + \frac{1}{\D t})}{\Gamma (k + 1 + \frac{1}{\D t} + 2 K_1)  \Gamma (1 +  \frac{1}{\D t})}
\nonumber \\
&&+ C^2 \D t \sum_{r=0}^{k-1} \frac{\Gamma (k + 1 + \frac{1}{\D t}) \Gamma (r + 1 + \frac{1}{\D t} + 2K_1)} {\Gamma (k + 1 + \frac{1}{\D t} + 2 K_1) \Gamma (r + 1 + \frac{1}{\D t})} (1 + r \D t)^{- 2 K_1}.
\end{eqnarray}
Using the similar argument as the proof of Theorem \ref{themEM} and the estimates in Lemma \ref{lemgamma}, one can obtain that
\begin{eqnarray*}
&&\frac{\Gamma (k + 1 + \frac{1}{\D t}) }{\Gamma (k + 1 + \frac{1}{\D t} + 2 K_1)} \frac{\Gamma (1 + 2 K_1 + \frac{1}{\D t})}{\Gamma (1 +  \frac{1}{\D t})}
\nonumber \\
&\leq& \frac{\Gamma (k + 1 + \frac{1}{\D t}) }{\Gamma (k + 1 + \frac{1}{\D t} + 2 K_1)}  \frac{\Gamma (1 +  \frac{1}{\D t} + 2K_1 - [2K_1])}{\Gamma (1 +  \frac{1}{\D t})} \prod_{i=1}^{[2K_1]} \left( 1+ \frac{1}{\D t} + 2 K_1 - i \right)
\nonumber \\
&\leq& \left( k+ 1 +\frac{1}{\D t} \right)^{-2 K_1}   \left( 1 +\frac{1}{\D t} \right)^{2 K_1 - [2 K_1]}  \left( 1 +\frac{1}{\D t} + 2 K_1 \right)^{[2 K_1]}
\nonumber \\
&\leq& ((k+1)\D t + 1)^{-2K_1} ((1+2 K_1)\D t + 1)^{2 K_1},  
\end{eqnarray*}
and
\begin{eqnarray*}
&&\frac{\Gamma (k + 1 + \frac{1}{\D t})}{\Gamma (k + 1 + \frac{1}{\D t} + 2 K_1)} \frac{\Gamma (r + 1 + \frac{1}{\D t} + 2K_1)}{\Gamma (r + 1 + \frac{1}{\D t})}
\nonumber \\
&\leq& \frac{\Gamma (k + 1 + \frac{1}{\D t})}{\Gamma (k + 1 + \frac{1}{\D t} + 2 K_1)} \frac{\Gamma (r + 1 + \frac{1}{\D t} + 2K_1 - [2K_1])}{\Gamma (r + 1 + \frac{1}{\D t})} \prod_{i=1}^{[2 K_1]} \left(r + 1 + \frac{1}{\D t} + 2 K_1 - i \right)
\nonumber \\
&\leq& \left(k + 1 + \frac{1}{\D t} \right)^{- 2 K_1} \left( r + 1 + \frac{1}{\D t} \right)^{2K_1 - [2 K_1]} \left(r + 1 + \frac{1}{\D t} + 2 K_1 \right)^{[2 K_1]}
\nonumber \\
&\leq& ((k+1)\D t + 1)^{- 2 K_1} ((r+1+ 2 K_1)\D t + 1)^{2 K_1}.
\end{eqnarray*}
Substituting these two estimates back to (\ref{BEMmain}), one can see that
\begin{eqnarray*}
\E |Z_k|^2 &\leq& ((k+1)\D t + 1)^{- 2 K_1} \bigg( ((1+2 K_1)\D t + 1)^{2 K_1} \E |Z_0|^2
\nonumber \\
&&+ C^2 \D t \sum_{r=0}^{k-1} ((r+1+ 2 K_1)\D t + 1)^{2 K_1} (1 + r \D t)^{- 2 K_1}\bigg).
\end{eqnarray*}
Since $((r+1+ 2 K_1)\D t + 1) / (1 + r \D t)$ can be bounded by some positive constant $C_2$, one has that 
\begin{eqnarray*}
\E |Z_k|^2 &\leq& ((k+1)\D t + 1)^{- 2 K_1} \bigg( ((1+2 K_1)\D t + 1)^{2 K_1} \E |Z_0|^2 + C^2 C_2^{2 K_1} k \D t \bigg)
\nonumber \\
&\leq& ((k+1)\D t + 1)^{- 2 K_1 + 1} (\E |Z_0|^2 + C^2 C_2^{2 K_1}).
\end{eqnarray*}
Then the assertion (\ref{BEMmeansquare}) holds. \eproof
\end{proof}

\begin{rmk}
One can notice that when applying Lemma \ref{lemgamma} to estimate the two terms on the right hand side of the inequality (\ref{BEMmain}) we only need $K_1 > 0.5$ and this is due to the term $2K_1$. The condition for $K_1$ here is in line with the that in Theorem \ref{themsde} for the SDE and better than the condition  for the EM method. 
\end{rmk}

\section{Numerical Examples}\label{numexp}

We first consider the equation that
\begin{equation*}
dx(t) = - \frac{x(t)}{1+t} dt + \frac{1}{1+t}dB(t).
\end{equation*}
It is easy to check that (\ref{lingrof}), (\ref{onesidf}) and (\ref{lingrog}) are all satisfied with $K_1 = 1$. Applying the EM method to the equation with step size 0.1, we simulate 1000 paths and plot the mean square of them on the left of Figure \ref{linearcase}.

\begin{figure}[ht]
\begin{center}$
\begin{array}{cc}
\includegraphics[width=3in]{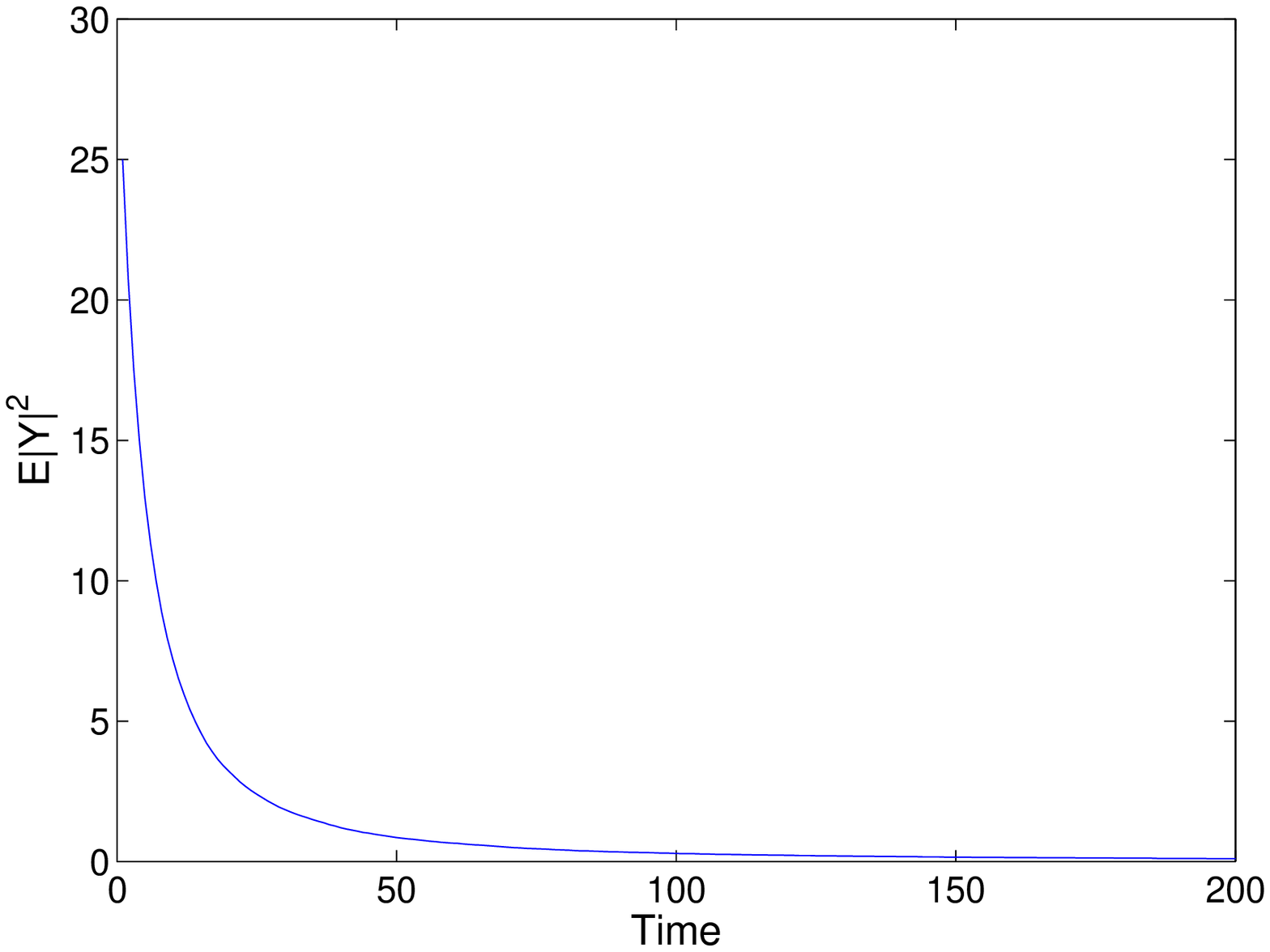} &
\includegraphics[width=3in]{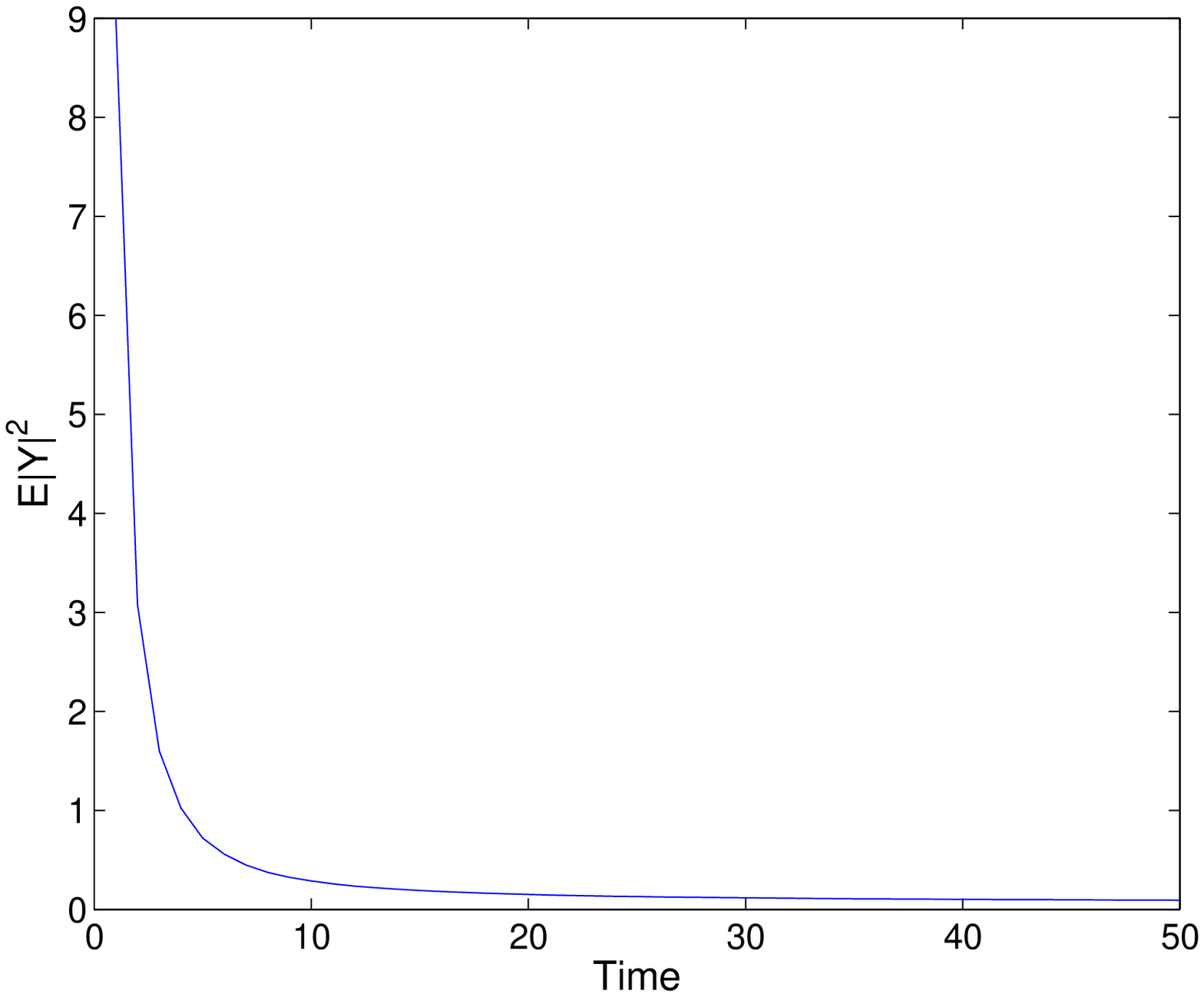}
\end{array}$
\end{center}
\caption{Left: mean square of the EM solution; Right: mean square of the BEM solution }
\label{linearcase}
\end{figure}
\noindent

Next we consider the following equation 
\begin{equation*}
dx(t) = \frac{-3 x(t)- x^3(t)}{(1+t)^2} dt + \frac{5 \sin(x(t))}{(1+t)^4}dB(t).
\end{equation*}
It can be seen that (\ref{lingrof}) is no longer satisfied, but (\ref{onesidf}) and (\ref{lingrog}) still hold with $K_1 = 3$. The mean square of the BEM solution calculated from 100 paths with step size 0.3 is plot on the right of Figure \ref{linearcase}.

\section{Conclusions and Future Research}\label{conclusion}
In this letter, the reproduction of the mean square polynomial stability of a class of SDEs is studied. Both the EM method and the BEM method are considered here. It is shown BEM method can recover the polynomial stability properties of a wider range of SDEs. Our techniques are based on properties of the gamma function and they are well suited for showing polynomial decay rates. 
\par
To our best knowledge, the polynomial stability appears more often in stochastic differential delay equations (SDDEs). But few works have been done on the numerical reproduction of the polynomial decay rate for SDDEs, particularly in the moment sense. Therefore, this could be one of the interesting future research.
\bibliography{references}
\end{document}